\documentclass[reqno,11pt]{amsart}
\pdfoutput=1		

\usepackage[bookmarksnumbered]{hyperref}	
\usepackage[ngerman,USenglish]{babel}			
\usepackage[ascii]{inputenc}								
\usepackage{aliascnt}										
\usepackage{microtype}									
\usepackage[all]{hypcap}
\usepackage{tikz}											
\usepackage{verbatim}

\newcommand{\newjointcountertheorem}[3]{
	\newaliascnt{#1}{#2}
	\newtheorem{#1}[#1]{#3}
	\aliascntresetthe{#1}	
}

\usepackage{amsmath}
\usepackage{amsfonts}     
\usepackage{longtable}
\usepackage{amssymb}
\usepackage{color}
\usepackage{graphicx} 
\def\a{{\alpha}}

\def\d{{\delta}}

\def\s{{\sigma^2}}
\def\phi{\varphi}
\def\eps{{\varepsilon}}
\def\qed{\hfill\vrule height5pt width5pt depth0pt}

\newcommand{\black}{\color{black}}

\newcommand{\Z}{\mathbb{Z}}

\newcommand{\EE}{\mathbb{E}}
\newcommand{\PP}{\mathbb{P}}

\newtheorem{thm}{Theorem}[section]
\newjointcountertheorem{satz}{thm}{Satz}
\newjointcountertheorem{lem}{thm}{Lemma}
\newjointcountertheorem{cor}{thm}{Corollary}
\newjointcountertheorem{prp}{thm}{Proposition}
\newjointcountertheorem{cnj}{thm}{Conjecture}
\newjointcountertheorem{que}{thm}{Question}
\newjointcountertheorem{fct}{thm}{Fact}
\theoremstyle{definition}
\newjointcountertheorem{dfn}{thm}{Definition}
\newjointcountertheorem{ntn}{thm}{Notation}
\newjointcountertheorem{rem}{thm}{Remark}
\newjointcountertheorem{nte}{thm}{Note}
\newjointcountertheorem{exl}{thm}{Example}
\newjointcountertheorem{rems}{thm}{Remarks}
\newjointcountertheorem{exls}{thm}{Examples}

\def\Snospace~{\S{}}

\renewcommand{\mathbb}{\mathbf}

\newcommand{\lra}{\longrightarrow}

\newcommand{\var}{{\mathrm{Var}}}

\numberwithin{equation}{section}
\allowdisplaybreaks[4]							

\begin{document}
\title[{Small drift limit theorems for occupation times of random walks}]{Small drift limit theorems for random walks}

	\author{Ernst Schulte-Geers and Wolfgang Stadje}
\address{Bundesamt f\"ur Sicherheit in der Informationstechnik (BSI), Godesberger Allee 185--189, 53175 Bonn, Germany} 
	\email{ernst.schulte-geers@bsi.bund.de}
	\address{Institute of Mathematics, University of Osnabr\"{u}ck, 49069 Osnabr\"{u}ck, Germany}
    \email{wstadje@uos.de}

\begin{abstract}
We show analogs of the classical arcsine theorem for the occupation time  of a random walk in $(-\infty,0)$ in the case of a small positive drift. 
To study the asymptotic behavior of the total time spent in  $(-\infty,0)$ we consider parametrized classes of random walks, where the convergence of the parameter to zero implies the convergence of the drift to zero. 
We begin with  shift families, generated by a centered random  walk by adding to each step a shift constant $a>0$ and then letting $a$ tend to zero. Then we study families of associated distributions. In all cases we arrive at the same limiting distribution, 
which is the distribution of the time spent below zero of a standard Brownian motion with drift 1. For shift families this is explained by a functional limit theorem. Using fluctuation-theoretic formulas we derive the generating function of the occupation time  in closed form, which provides an alternative approach. 
 In the course also give a new form of the first arcsine law for the Brownian motion with drift.

	
\end{abstract}

	\maketitle

	
\section{Introduction}
For the classical symmetric random walk with $\pm 1$ steps it is well known that the three random variables ``time spent on the positive axis", ``position of the first maximum"
and ``last exit from zero'' are identically distributed and (suitably normalized)  asymptotically arcsine-distributed. 
Here the norming factor is the length of the time interval the random walk has been observed, so that the limiting statements refer to ``relative" times. 

Consider now a classical random walk with drift $\d\neq 0$. Clearly the same ``relative" variables can be studied. The asymptotic distribution of the random variable 
``(fraction of) time spent in $(-\infty,\a]$
has been determined by Tak\'acs \cite{Tak}, by applying a functional limit theorem. 

But if $\d\neq 0$ there is also another, ``absolute'' perspective. If for example $\d > 0$ for a general random walk, 
it is clear that $Z(\d)=$ ``number of visits in $(-\infty,0)$ ''  is almost surely finite, and that $Z(\d)\lra \infty$ in probability as $\d\searrow 0$.
One may ask  if $Z(\d)$, after multiplication with some deterministic function $a(\d )$, has a non-degenerate limit distribution. 
This paper aims to answer these and related questions for random walks in the heavy-traffic regime, i.e., when the drift converges to zero. \black  In all cases the limiting distribution for the occupation time in $(-\infty, 0)$, properly rescaled, turns out to have the density 
\begin{align}
\label{density} 
p(t)=2\,{\phi(\sqrt{2t})\over \sqrt{2t}}-2\,\Phi(-\sqrt{2t}),\ \ t>0 
\end{align} 
where $\phi$ and $\Phi$ are the density and the distribution function of $N(0,1)$, respectively. 

The distribution of the occupation time  in $(-\infty,0)$ \black of Brownian motion with  positive \black drift also has density (\ref{density}), and  in Section 2 
we begin with related results for Brownian motion. We  show \black for example  that \black the  distribution of the  time of the last exit from $0$ \black
of Brownian motion with drift
during a finite time interval is composed of the arcsine and a truncated exponential distribution. 
In Section 2 we derive the limiting occupation time distribution for shift families generated from a centered random  walk by adding to each step a shift constant $a>0$ and then letting $a$ tend to zero. The proof that (\ref{density}) gives the asymptotic distribution is based on Donsker's invariance principle. 
In Section 4 we give the key fluctuation-theoretic formulas for the distribution for the occupation time in $(-\infty, 0)$ for general random walks. 

The arcsine law and its ramifications are a classical topic but there are always recent contributions, for example some new explicit distributions \cite{L}, new proofs 
\cite{Hof}, or asymptotic considerations \cite{Mar}. Interesting results on the number of visits to one point by skipfree random walks and related questions can be found in \cite{Ross}. The problem considered in this paper is also connected to the heavy traffic approximation problem in queueing theory, in which the growth of  the all-time maximum of $S_n-na$ (where $S_n$ is the $n$th partial sum of iid random variables with mean zero) is studied as $a\searrow 0$. In the queueing context this is equivalent to the growth of the steady-state waiting time in a $GI/G/1$ system when the traffic load tends to 1. 
This question was first posed by Kingman (see \cite{Kin}) and was investigated by many authors (e.g. \cite{BoxCoh, Kos, Pro, ResSam, Shnwa}). 

\section{Occupation times and last exit from $0$ \black for Brownian motion with drift}
We start by presenting two results on occupation times for Brownian motion  with positive drift $\d>0$ and variance $\s$, one known and one new. 
Let $B_t$ be a standard Brownian motion and $X_t=\sigma B_t+\d t$.

\begin{lem}\label{lem}  
(1) Let $z>0$ and $T_z=\inf\{ t\geq 0\;:\;X_t\geq z\}$ be the first time when $X_t$ reaches level $z$.
Then $T_z$ has Laplace transform
$$\ell_{T_z}(s)=\EE e^{-sT_z}=\exp\left(-{ z \over \s}(\sqrt{\d^2+2\s s}-\d)\right).$$
(2) Let $V_0=V_0(\d)=\int_{0}^\infty 1_{(-\infty,0)}(X_t)\,dt$ be the total time that $X_t$ spends below zero.
Then $V_0$ has Laplace transform
$$\ell_{V_0}(s)=\EE e^{-sV_0}={ 2\d \over \d +\sqrt{\d^2+2\s s}}.$$
\end{lem}
Proofs for (1) resp. (2) (for $\s=1$) can be found in \cite{KarT} resp. \cite{Imh}. 
Note $(\d^2/2\s) V_0$ has the Laplace transform $ 2/(1 +\sqrt{1+ s})$. We call $A$ a generic random variable with this  
Laplace transform. 

The density of $A$ is given by (\ref{density}). 
To see this, note that $1 / \sqrt{1+s}$ is the Laplace transform of the gamma distribution $\Gamma_{1,{1\over 2}}$, which has density 
$$\gamma_{1,{1\over 2}}(t)=1_{(0,\infty)}(t)\,{e^{-t} \over \sqrt{\pi t}}.$$
Therefore $[1-(1 / \sqrt{1+s})]/s$ is the Laplace transform of $1-\Gamma_{1,{1\over 2}}(t)=\int_t^\infty \gamma_{1,{1\over 2}}(x)\,dx$. The equality 
$${1 \over 1+\sqrt{1+s}}={1 \over \sqrt{1+s} }- {1 \over s}\Big(1-{1\over \sqrt{1+s}}\Big)$$
now yields density (\ref{density}). 

For $z\geq 0$ let $V_z = 
\int_{0}^\infty 1_{(-\infty,z)}(X_t)\,dt$ the total time the process spends below $z$.
Then the obvious decomposition (obtained by conditioning on $T_z$ ) $V_z=T_z+V_0^\prime$ (where $V_0^\prime $ is independent of  $T_z$ and dsitributed as $V_0$) yields

\begin{lem}\label{lem2} $V_z$ has Laplace transform
$$\ell_{V_z}(s)=\EE(e^{-sV_z})=\ell_{T_z}(s)\,\ell_{V_0}(s).$$
\end{lem}
The density and distribution function are given in \cite{Imh}.\\

We focus in the sequel on the time spent on  the negative axis, but it is also of interest to  look at the other classical arcsine variable, i.e., the 
time of the last exit from $0$. Here we determine its distribution. \black   
 Let $\d \in\mathbb{R}\setminus\{0\}, \s =1$, so that $X_t=B_t+\d t$, and consider  $W=\sup\{t\in [0,1]\,:\,X_t=0\}$, the last time
$X_t$ visits $0$ in $[0,1]$. 

Recall that for $\d=0$, i.e., for the standard Brownian motion, the standard arcsine distribution (which has density $1_{(0,1)}(t)(1 / \sqrt{t(1-t)})$ 
and distribution function $(2 / \pi)\arcsin(\sqrt{t})$ on $[0,1]$) is the distribution of the last exit time from zero in the interval $[0,1]$. 

\black The distribution of $W$ turns out to have a nice representation in terms of the standard arcsine distribution
and a truncated exponential distribution.  
As this result seems new, we provide a proof. 
 
\begin{thm}\label{thmle} $W\stackrel{d}{=}C\cdot \min \{1,D_\d\}$ where $C$ and $D_\d$ are independent, $C$ is arcsine-distributed, and $D_\d$ is $\exp(\d^2/2)$-distributed. 
The moments of $W$ are given by 
$$\mathbb{E}W^k=  {2k \choose k}{1 \over 2^{2k}}\,\int_0^1 ky^{k-1} e^{-\d^2y/2}\,dy, \  \ k\ge 1.$$
%
\end{thm}

\begin{proof} We use a random walk approximation in the style of Tak\'acs \cite{Tak}.
Let $Y_1,Y_2,\ldots$ be iid with $$\PP(Y_i=1)=p={1 \over 2}+{\d \over 2\sqrt{n}}, \ \ \PP(Y_i=-1)=q=1-p$$ 
($p$ and $q$ depend on $n$, but this is suppressed in the notation) and partial sums
$S_0=0,\,S_k=\sum_{i=1}^k Y_i$. 

It it easy to see that the  processes $X^{(n)}$ defined by 
$$X^{(n)}(t)={1\over \sqrt{n}} S_{\lfloor nt \rfloor}, \ \ 0\le t \le 1 $$ 
converge in distribution to $X=(X_t)_{t\in [ 0,1]}$ in $D[0,1]$. 

Furthermore, the last-exit time from $0$ is continuous in the Skorohod topology on $D[0,1]$ 
on a set of $P_X$-measure 1, and
 $$T_n=\sup\{t\in[0,1]\;:\;X^{(n)}(t)=0\} ={1\over n} \max\{0\leq k\leq n\;:\; S_k=0\} =:{M_n / n}\black$$
 Then it suffices to show that ${M_N/ N}\longrightarrow C\cdot \min\{1,D_\d\}$ as $N\longrightarrow \infty$. 
 
 Since ${1/ \sqrt{1-4pqz^2}}$ and $(\sqrt{1-4pqz^2})/(1-z)$ are the generating functions for the sequences of probabilities $\PP(S_n=0)$ and $\PP(S_1\not = 0,\ldots,S_n\not = 0)$, \black
 respectively, 
the generating function of $M_N$ is 
\begin{align*}
\mathbb{E}t^{M_N}  &\ =\sum_{k=0}^N t^{k}\, \PP(S_k=0, S_{k+1} \not=0,\ldots,S_{N}\not=0)\\
                   & =\sum_{k=0}^N t^{k}\, \PP(S_k=0)\,\PP(S_{1} \neq 0,\ldots,S_{N-k}\neq 0)\\ 
                   &=[z^{N}] {1\over \sqrt{1-4pqt^2z^2}}{\sqrt{1-4pqz^2} \over 1-z}\\
                   &=[z^{N}] {1\over \sqrt{1-4pqt^2z^2}}{\sqrt{1-4pqz^2} \over 1-z^2}(1+z). 
\end{align*}
(Here and in the following $[z^{N}] f(z)$ denotes the coefficient of $[z^{N}]$ in the Taylor expansion of the function $f(z)$ around zero.) 
Thus the generating functions for $N=2n+1$ and $N=2n$ are identical and it is enough to consider even $N$.
Let $N=2n$ be even (and  $n>\d^2$) and $U_n={M_N / 2}$. 
Then the generating function of $U_n$ is
\begin{align*}
\mathbb{E}t^{U_n}&=[z^{2n}] {1\over \sqrt{1-4pqtz^2}}{\sqrt{1-4pqz^2} \over 1-z^2}\\
                   &=[z^{n}] {1\over \sqrt{1-4pqtz}}{\sqrt{1-4pqz} \over 1-z}
\end{align*}
so that the $k-$th factorial moment $u_{k,n}=\mathbb{E}\left(U_n(U_n-1)\cdots(U_n-k+1)\right)$ of $U_n$ is
given by 
\begin{align*}
u_{k,n}&= k! (-1)^k {-\frac{1}{2} \choose k}(4pq)^k\,[z^{n-k}]{1\over ({1-4pqz})^k\,(1-z)}\\
       &= k (-1)^k {-\frac{1}{2} \choose k}(4pq)^k\,[z^{n-k}]{1\over (1-z)}\int _{0}^{\infty} x^{k-1} e^{-(1-4pqz)x}\,dx\\
       &= (-1)^k {-\frac{1}{2} \choose k}(4pq)^k\,\int _{0}^{\infty} kx^{k-1} e^{-x}\left(\sum_{j=0}^{n-k}\frac{(4pqx)^j}{j!}\right)\,dx. 
\end{align*}
Now denote by $Poiss (\lambda)$ a random variable having  the Poisson  distribution with parameter $\lambda$. As $4pq=1-({\d^2}/{2n})$, we obtain 
\begin{align*}
&\int _{0}^{\infty} k x^{k-1} e^{-x}\left(\sum_{j=0}^{n-k}\frac{(4pqx)^j}{j!}\right)\,dx\\
&=\int _{0}^{\infty} k x^{k-1} e^{-x(1-4pq)}\,\mathbb{P}\left(Poiss(4pqx)\leq n-k \right)\,dx\\
&=n^k\,\int _{0}^{\infty} k y^{k-1} e^{-\d^2 y/2}\,\mathbb{P}\left(Poiss((n-\frac{\d^2}{2})y)\leq n-k\right)\,dy. 
\end{align*}
By the central limit theorem,  
$$\mathbb{P}\left(Poiss((n-\frac{\d^2}{2})y)\leq n-k\right)\longrightarrow\left\{
\begin{array}{cll}  &1  &\mbox{ for } 0\leq y<1\\ &\frac{1}{2} &\mbox{ for } y=1 \\ &0 &\mbox{ for } y>1
\end{array}\right.$$
so that for every $k$ we have 
$$\frac{u_{n,k}}{n^k}\longrightarrow (-1)^k {-\frac{1}{2} \choose k} \int_{0}^1 ky^{k-1}e^{-y\d^2/2}\,dy.$$ 
 Hence ${\mathbb{E}T_N^k}/{N^k}$
 tends to the same limit. This shows the second assertion. Finally,  $$\EE C^k = (-1)^k {-\frac{1}{2} \choose k}={2k \choose k} {1 \over 2^{2k}}$$  and 
integration by parts shows that $\int_{0}^1 ky^{k-1}e^{-y\d^2/2}\,dy=\mathbb{E} \min\{1,D_\d^k \}$.
Thus all moments of ${T_N / N}$ converge to the corresponding moments of $C\cdot\min\{1,D_\d \}$. Since the distribution of 
$C\cdot\min\{1,D_\d \}$ is clearly determined by its moments the first assertion follows.
\end{proof} 

\begin{rem} 
As an immediate consequence of the scaling properties of Brownian motion we see that  
the distribution of 
$$W_T=\sup\{ t\leq T\,:\, \sigma B_t+\d t=0\}$$ is the same as that of $C\cdot\min\{T,D_{\d/\sigma} \}$. 
The time of the last zero of $\;\sigma B_t+\d t\;$ in the interval
$[0,\infty)$ is thus distributed as $C\cdot D_{\d/\sigma} $, which is the gamma distribution with parameters ${\d^2 / 2\sigma^2}$ and ${1}/{2}$.
\qed 
\end{rem}

\begin{rem} Clearly $V_0$ (the occupation time on the negative axis) is stochastically smaller than $W_\infty$ (the last exit time from zero), and the results above quantify this
precisely. We find e.g. that $$\EE(V_0)=\dfrac {\s}{2\d^2}=\dfrac{1}{2}\,\EE(W_\infty).$$  
\qed 
\end{rem}
\black

\begin{rem} Last-exit times of Brownian motion from moving boundaries have been studied intensively,
and  more complicated  expressions for the density of the last-exit time from a linear boundary were derived in \cite{Sal} and \cite{Imh2}. The representation in (\ref{thmle})  appears to be new, as it
is not mentioned in the encyclopedic monograph \cite{BorS}. For the density of the sojourn time found by Tak\'acs by a random walk limit two ``purely Brownian" explanations have 
been given in \cite{DonY}. It is natural to ask for such an explanation for the representation in (\ref{thmle}).
  \qed  
\end{rem}

\section{Limit of occupation times for shifted random walks}

In this section we consider a shifted random walk. Specifically, let $(X_{\d,1},X_{\d,2},\ldots)$ be a  parametrized sequence
of iid random variables with $\EE(X_{\d,i})=0$, $\var(X_{\d,i})=\s(\d) \in (0, \infty)$. Let $\d>0$ and 
$$ Y_i^\d =X_{\d,i}+\d, \ \ 
  S_{\d,n}=\sum_{i=1}^n X_{\d,i}, \ \ 
 S_n^\d=\sum_{i=1}^n Y_i^\d.$$ 
We are interested in the occupation time 
  $$Z_0^\d=\sum_{i=1}^\infty 1_{(-\infty,0)}(S_n^\d).$$ 
Throughout this section we assume that $\s(\d)\lra \s>0$ as $\d\lra 0$ and that 
 the following Lindeberg-type condition holds:  for every $\eps>0$,
\begin{equation}\label{lind}
\lim_{\d\lra 0} \int_{|\d X_{\d,1}|>\eps} X_{\d,1}^2\,d\PP=0.
\end{equation} 
These conditions are chosen such that for the triangular array with the variables 
$$Z_{\d,k}=\dfrac{\d}{\sigma(\d)}X_{\d,k}, \ \ k=1,\ldots,\lfloor\dfrac{1}{\d^2}\rfloor $$ 
the central limit theorem holds: indeed, $\var(Z_{\d,1})=\d^2$ and the Lindeberg condition for this triangular array reads as 
$$\lim_{\d\lra 0} \frac{1}{\d^2}\int_{|Z_{\d,1}|>\eps\d^2\lfloor\frac{1}{\d^2}\rfloor} Z_{\d,1}^2\,d\PP=
\lim_{\d\lra 0} \frac{1}{\s(\d))}\int_{|\d X_{\d,1}|>\eps \sigma(\d) \d^2\lfloor\frac{1}{\d^2}\rfloor} X_{\d,1}^2\,d\PP=
0 \ \  \text{for every} \  \eps>0,   
$$ 
which is clearly true under the conditions above. 

We use similar ideas as Prohorov \cite{Pro}, who proved the following:   

\begin{thm} (Prohorov) In the situation above let
$M^\d =\min\{ S_n^\d\;:\;n\geq 0\}$. Then 
$$\PP(\d M^\d>x)\lra e^{-2x/\s} \ \mbox{ for all } x>0.$$ 
\end{thm} 
In \cite{Pro} 
 the maximum in the case of negative drift was considered instead of $M^\d$. The result had been proved earlier 
by Kingman under the assumption of the existence of an exponential moment.

The following lemma will be needed to obtain tightness bounds.
\begin{lem}
 In the situation above let $z\geq 0$ and let $\d_k>0$ be a sequence of positive numbers satisfying $\sup_{k\geq 1} \s(\d_k)< \infty$.
 Then for every $\eps>0$ we can find a $T$ such that for all $k$
$$\PP(\sup_{n\geq T/\d_k^2} (|S_{\d_k,n}|-n\d_k)\geq -\frac{z}{\d_k})<\eps.$$ 
\end{lem}
 \begin{proof} First consider a sequence $S_n$ of partial sums
of an arbitrary iid sequence $(X_i)$ with $\EE(X_1)=0$ and $\var(X_1)=\s$. Let $a,b>0,Na>b$ and consider the event $E_N =\{\sup_{n\geq N} (|S_n|-na)\geq -b\}$. 
Clearly \begin{align*}E_N&=\bigcup_{j=0}^\infty\left\{\max_{2^jN\leq n< 2^{j+1}N} (|S_n|-na)\geq -b\right\}\\
                         &\subseteq \bigcup_{j=0}^\infty\left\{\max_{2^jN\leq n< 2^{j+1}N}|S_n|\geq 2^jNa- b \right\}\\
                         &\subseteq \bigcup_{j=0}^\infty\left\{\max_{n\leq 2^{j+1}N}|S_n|\geq 2^jNa- b \right\}.\\
\end{align*}
By Kolmogorov's inequality,  
$$\PP(\max_{n\leq 2^{j+1} N}|S_n|\geq 2^jNa-b)\leq \frac{2^{j+1} N\s}{(2^jNa-b)^2}. $$
Now set $N=\frac{T}{\d^2}, a=\d,  b=\frac{z}{\d}, X_i=X_{\d,i}$. It follows that 
$$\PP(\sup_{n\geq T/\d^2} (|S_{\d,n}|-n\d)\geq -\frac{z}{\d})\leq\sum_{j=0}^{\infty}\frac{2^{j+1} T \s(\d)}{(2^jT-z)^2}.$$
The bound on the right side depends on $\d$ only via $\s(\d)$ and can clearly be made arbitrarily small (under the assumptions above).
\end{proof}

\begin{cor}\label{cor1}
 In the situation above let $z\geq 0$ and let $\d_k>0$ be a sequence of positive numbers satisfying $\sup_{k\geq 1} \s(\d_k)<\infty$. 
Then for every $\eps>0$ one can find a $T$ such that for all $k$
$$\PP(\min_{n\geq T/\d_k^2} \d_k(S_{\d_k,n}+n\d_k)\leq z)<\eps. $$ 
\end{cor}

\begin{thm}\label{thmconv}
 $${\d^2 \over 2\s(\d)\black}Z_0^\d \lra A \ \mbox{ in distribution} \mbox{ as }\d\searrow 0. $$
\end{thm}
\begin{proof}
By the remark following \ref{lem} it suffices to show that $\d^2 Z_0^\d\lra V_0$ in distribution, where $V_0$ is the distribution of the time
the process $W_t=\sigma B_t+t$ spends below zero. 

Let $T>0$ and consider the sequence of processes $$U^\d(t)=\d\,\sum_{i=1}^{\lfloor t/\d^2\rfloor}Y_i^\d, \;\;\;0\leq t \leq T.$$
By Donsker's limit theorem (in the version for triangular arrays, see e.g. \cite{Bil}, p.147), 
the sequence $U^\d\lra \sigma B + id$ in distribution in $D[0,T]$, where $\sigma B + id$
denotes the Brownian motion with variance $\s$ and drift $1$, i.e., with coordinate variables $\sigma B_t +t$. For any bounded Borel function
$v$ on $[0,T]$ the functional 
$x\mapsto \int_{0}^T v(x_t)  \,dt$ on $D[0,T]$ is Skorohod-measurable and continuous except on a set of $B$-measure $0$ (see e.g. \cite{Bil}, p. 247).
Thus, 
\begin{align*}
\d^2 \mbox{card}(\{n\; :\;  S_n^\d < 0, 1\leq n \leq T/\d^2\})&=\int_0^{\d^2\lfloor T/\d^2\rfloor} 1_{(-\infty,0)}(U^\d(t))\,dt\\
 & \lra \int_0^T 1_{(-\infty,0)}(X_t)\,dt \mbox{ as }\d\searrow 0\,
\end{align*} 
in distribution and we will be done if we can justify the interchange of the limits $T\lra \infty$ and $ \d \searrow 0$.
Let $\d_k > 0$ be a sequence decreasing to zero and let $\eps>0$. By corollary \ref{cor1} 
we can find an $N$ such that $\PP(\min_{n\geq N/\d_k^2} S_n^{\d_k}\leq 0)<\eps$ for all $k$.
Thus,  
\begin{equation}\label{lim}\lim_{T\lra \infty} \sup_{k\ge 1}\PP(\min_{n\geq T/\d_k^2} S_n^{\d_k}\leq 0)=0
\end{equation} 
and the assertion follows since, by the monotone convergence theorem,  
$$ \lim_{T\lra \infty}\int_0^T 1_{(-\infty,0)}(X_t)\,dt = \int_0^\infty 1_{(-\infty,0)}(X_t)\,dt.$$ 
\end{proof}

\begin{rem}
A related discussion can be found in \cite{Shnwa}. In that paper, Shneer and Wachtel derived an extension of Kolmogorov's inequality and treated the maximum of random walks with negative drift and step size distributions attracted to a stable law of index $\a\in(1,2]$.
In the case of finite variance ($\a=2$) they already remarked that their results (including in particular the crucial relation \eqref{lim}) remain valid 
  if the conditions assumed above hold.
\qed 
\end{rem} 

\begin{rem} 
Assume that the $X_i$ are independent with $\EE(X_i)=0$ and variances $\var(X_i)=\sigma_i^2$
and satisfy Lindeberg's condition. Let $s_i^2=\sum_{k=1}^i \sigma_k^2$. Then the step processes
$X_n(t)$ which jump to the value ${S_i / s_n}$ at time ${s_i^2 / s_n^2}$ converge weakly to
a standard Brownian in $D[0,1]$ (by Prohorov's extension of Donsker's theorem). One may thus expect that they exhibit a similar limiting behavior.
\qed 
\end{rem} 

 Finally, replacing $0$ by $z/\d$ and repeating the steps in the proof of \ref{thmconv} yields 

\begin{thm}\label{thmz}
In the situation above let $z>0$ and $Z_z^\d=\sum_{n=1}^\infty 1_{(-\infty,z)}(S^\d_n)$ . Then 
$\d^2 Z^\d_{z/\d} \lra V_z$ in distribution, where the Laplace transform of $V_z$ is given in Lemma \ref{lem2}  
with $\d=1$. 
\end{thm} 

 If here $z$ depends on $\d$ such that  $\d z(\d)\lra 0$ as $\d\lra 0$ we find \black

\begin{prp} In the situation above let $(z(\d))$ a sequence of positive numbers with $z(\d)=\mathrm{o}(1/\d)$ and $\sup_\d z(\d)<\infty$. Then 
	$$\d^2 Z^\d_{z(\d)}\lra V_0=2\s A \mbox{  as  } \d \lra 0.$$
\end{prp}
\begin{proof} Clearly 
$V_0$ is stochastically smaller than any distributional limit of $\d^2 Z^\d_{z(\d)}$ 
(because $Z_0^\d$ is stochastically smaller than $Z_y^\d$ for $y\geq 0$),
furthermore $V_y=T_y+V_0$ is stochastically smaller than $V_z$ for $y\leq z$. 
Let $\eps>0$ and $C=\sup_\d z(\d)$, then $C<\infty$ and $\d^2 Z^\d_{\eps C/\d}\lra V_{\eps C}$ in distribution as $\d \lra 0$ (by theorem \ref{thmz}).
Since $Z^\d_{z(\d)}=Z^\d_{\eps z(\d)/\eps}$ is stochastically smaller than $Z^\d_{\eps C/\d}$ for $\d\leq \eps$, 
any distributional limit of $\d^2 Z^\d_{z(\d)}$ is stochastically smaller than $V_{\eps C}$. Thus the distributional limit
exists and equals  $V_0$.
\end{proof}

We close this section with an application of Theorem \ref{thmconv} 
in a frequently encountered situation.

\begin{exl}\label{exl1} (Expectation shift in exponential families.)\\
Let $U$ be a non-constant real random variable such that the moment generating function $$m(s)=\EE e^{sU}$$ is finite in an open interval $I$ around  $0$, and $E(U)=m^\prime(0)=0$,
$\mathrm{Var}(Y)=\s$. 
 
For $p\in I$ let $U_p$ have the ``associated'' distribution with moment generating function $m_p(s)={ m(p+s) \over m(p)}$, clearly $U_p$ has expectation 
$\EE(U_p)={ m^\prime(p) \over m(p)}$ and variance $\s(p)={ m^{\prime\prime}(p)m(p)- (m^\prime(p))^2 \over m(p)^2}$.\\ Let $Z_0(p)$ denote the random variable
``time spent in $(-\infty,0)$'' by the random walk generated by iid variables with distribution $U_p$. Then 
$${(\EE(U_p))^2 \over 2 \s(p)}\,Z_0(p)\lra  A  \mbox{ in distribution for } p\searrow 0 \;\;\;.$$
\begin{proof}
It is well known that $s\mapsto \log m(s)$ is strictly convex on $I$, thus $p\mapsto {m^\prime(p) \over m(p)}=\EE(U_p)$ is strictly increasing. Thus we may parameterize the
distributions by $\d(p)= \EE(U_p)$. We have  $\d(p)\searrow 0$ for $p \searrow 0$ and 
$\s(p)\lra \s$ as $p\searrow 0$. Let $X_{\d(p)}=U_p-\EE(U_p)$ and $Y^{\d(p)}=X_{\d(p)}+\d(p)=U_p$.
Then the Lindeberg condition \eqref{lind} is satisfied, since by Chebyshev's inequality   
$$\int_{|\d(p) X_{\d(p)}|>\eps} X_{\d(p)}^2\,d\PP\leq {\d^2(p)\s(p)\over \eps^2}$$
and the claim follows from Theorem \ref{thmconv}. 
\end{proof}

\end{exl}

\black

\section{The fluctuation theoretic approach}

The topics investigated here belong to the fluctuation theory of random
walks. We recall some basic facts, which will be used in 
the sequel and can e.g. be found in Section XII.7 of \cite{Fel2}. 

We consider a random walk $(S_n)_{n\ge 1}$, i.e., a sequence of partial sums of iid random variables 
and let $R=\inf\{ n\geq 1\;:\; S_n<0\}$ and  $W=\inf\{n \geq 1\;:\; S_n\geq 0\}$ be the lengths of the 
first strictly descending and weakly ascending ladder epochs of the random walk, 
respectively.We denote by $r(z)$ and $a(z)$
denote the corresponding  probability generating functions  and set $\mu = \EE W $.   
The occupation time of interest is $Z_0=\sum_{n=1}^\infty 1_{(-\infty, 0)}(S_n)$. 
\begin{thm}\label{thmSpA}(Sparre Andersen)
For $|z|<1$
\begin{align*}
{1 \over 1-r(z)}&=\exp \left\{\sum_{n=1}^\infty {z^n \over n} \PP(S_n<0)\right\} \mbox{  }\\ 
{1 \over 1-a(z)}&=\exp \left\{\sum_{n=1}^\infty {z^n \over n} \PP(S_n \ge 0)\right\} \mbox{  }\end{align*}
\end{thm}

An immediate consequence is the factorization theorem. 
\begin{thm}(Duality) 
For $|z|<1$
\begin{align*}
(1-r(z))(1-a(z))=1-z. 
\end{align*}
\end{thm}

It follows from the factorization theorem is that $W$ ($R$) has a finite expected value if and only if $R$ ($W$) is defective, and that
the relations $\EE(R)\PP(W=\infty)=1$ and $\EE(W)\PP(R=\infty)=1$ hold.

At the combinatorial heart of fluctuation theory is the ``Sparre Andersen transformation'' (made explicit by Feller and 
refined by Bizley and Joseph) given in Lemma 3 of XII.8 of \cite{Fel2}:

\begin{lem} Let $x_1,\ldots,x_n$ be real numbers with exactly $k\geq 0$ negative partial sums
$s_{i_1},\ldots,s_{i_k}$, where $i_1>\ldots> i_k$. Write down $x_{i_1},...,x_{i_k}$ followed
by the remaining $x_i$ in their original order. (If $k=0$, the sequence remains unchanged).  The transformation thus defined is invertible, and
the first (absolute) minimum of the partial sums of the new arrangement  
occurs at the $k$-th place.
\end{lem}

Clearly this extends to {\it infinite sequences} with exactly $k$ negative partial sums: just apply the bijection above 
to an initial section large enough to contain all the negative partial sums, and leave the rest unchanged.

The following formulas express 
the generating function of $Z_0$ in terms of $r(z)$ or of $a(z)$, respectively.

\begin{thm}
\begin{align}\label{asc}
\EE z^{Z_0}={1-r(1) \over 1-r(z)}  = {1 \over \mu}{ 1-a(z) \over 1-z}
 =\exp\left\{-\sum_{k=1}^\infty (1-z^k)\frac{\PP(S_k<0)}{k} \right\}. 
\end{align} 
\end{thm}

\begin{proof}
According to Lemma 2.3, each sequence $x_1,x_2\ldots$ with exactly $k$ negative partial sums there corresponds (by a finite reordering) a unique sequence with first (absolute) minimum  at the $k$th
place. The partial sums $s_0=0,s_1,s_2,\ldots$ of the rearranged sequence consist of a first part $s_0,s_1,\ldots,s_k$ 
and a second part $s_{k+1},s_{k+2},\ldots $ such that  the partials sums satisfy $s_i>s_k$ for $i\leq k$ and $s_i-s_k\geq 0$ for $i>k$. 
For a random walk the joint distribution of the $X_i$ is invariant under finite permutations, and the two parts are independent. The first part has probability $$\PP(0>S_k,S_1>S_k,\ldots,S_{k-1}>S_k)=\PP(S_1<0,\ldots,S_k<0)$$ 
(by reversing the order of the variables), the second part has  probability 
$$\PP(S_{k+1}-S_k \geq 0,S_{k+2}-S_k\geq 0,\ldots)=\PP(S_1\geq 0,S_2\geq 0,\ldots)=1-r(1).$$
This yields the first equation of (\ref{asc}). 
The second   
one follows immediately from the factorization identity $(1-a(z))(1-r(z))=1-z$ (recall Theorem 4.2) and the third one from Sparre Andersen's theorem. 
\end{proof}

In some cases $r(z)$ can be computed in closed form, and the asymptotics of $Z_0 $ can be obtained from an explicit formula. 
An example is the normal random walk.
Let the iid steps $X_i$ be $N(\d, \s)$-distributed. Here we only assume  
that $\d \neq 0$, i.e., we consider the cases of positive and negative $\d$ simultaneously and let $d:=|\d|, q:=\frac{\d^2}{2\s}$. 
\begin{exl} For the normal random walk we have\\
(a) $ r(z)=1-(1-z)^{{1 \over 2}}\,\exp \left(\mbox{ sign}(\d)  {d^2 \over \pi\s} {\displaystyle \int_{0}^{1}\int_0^{\infty} \,{ e^{-d^2(y^2+x^2) / {2\s}} \over 1-ze^{-d^2(y^2+x^2) / {2\s}}}}\,dy\,dx \right). $ 

(b) $qZ_0\longrightarrow A$ in distribution as $\d^2/\s\searrow 0,\d \searrow 0$.  

(c) $r(e^{-qs})^{1/\sqrt{q}}\longrightarrow e^{-(\sqrt{1+s} -1)}$ as $q \searrow 0,\d \nearrow 0$.
\end{exl}
Note that here $\s$ may vary with $\d$, it is only essential that $\d/\sigma \longrightarrow 0$.
\begin{proof} 
Directly from Sparre Andersen's theorem we find that 
\begin{align*}
\log\left({1 \over 1-r(z)}\right)&=\sum_{n=1}^\infty {z^n \over n} \PP(S_n<0) 
                                    =\sum_{n=1}^\infty {z^n \over n} \int_{-\infty}^{-n\d} {1 \over \sqrt{2n\pi\s}}\, e^{-x^2 / {2n\s}}\,dx
\\ &                                    =\sum_{n=1}^\infty {z^n \over n}\left({1 \over 2} -   \mbox{sign}(\d)\, \int_{0}^{n d} {1 \over \sqrt{2n\pi\s}}\, e^{-x^2 / {2n\s}}\,dx\right).
\end{align*} 
Hence, 
\begin{align}\label{normal}
 1-r(z)=(1-z)^{{1 \over 2}}\,\exp(\mbox{ sign}(\d) G(z)),
 \end{align} 
where 
$$G(z)=\sum_{n=1}^\infty {z^n \over n}\, \int_{0}^{nd} {1 \over \sqrt{2n\pi\s}}\, e^{{-x^2 / 2n\s}}\,dx.$$
We have 
\begin{align*}
\int_{0}^{nd} {1 \over \sqrt{2n\pi\s}}\, e^{-x^2 / {2n\s}}\,dx&= \int_{0}^{d} \sqrt{ {n \over 2\pi\s}}\, e^{-ny^2 / {2\s}}\,dy
 =
{n \over \pi\s} \int_{0}^{d}\int_0^{\infty} \, e^{-n(y^2+x^2) /{2\s}}\,dy\,dx
\\ 
&= {nd^2 \over \pi\s} \int_{0}^{1}\int_0^{\infty} \, e^{-nd^2(y^2+x^2) / {2\s}}\,dy\,dx 
\end{align*}
and therefore  
$$G(z)={d^2 \over \pi\s} \int_{0}^{1}\int_0^{\infty} \,{ e^{-d^2(y^2+x^2) / {2\s}} \over 1-ze^{-d^2(y^2+x^2) / {2\s}}}\,dy\,dx, $$
proving (a). Note that $G(z)$ depends only on the ratio $q={d^2/ 2\s}$. Fix $s>0$. Setting $z=e^{-qs}$ we obtain for $q\searrow 0$ (by dominated convergence):
\begin{align*} 
G(e^{-qs})&={2\over \pi} \int_{0}^{1}\int_0^{\infty} \,{ q e^{-q(y^2+x^2) } \over 1-e^{-q(s+y^2+x^2) }}\,dy\,dx \\ & 
 \longrightarrow 
{2\over \pi} \int_{0}^{1}\int_0^{\infty} \,{ 1 \over s+y^2+x^2}\,dy\,dx 
=\log \left({1+\sqrt{1+s}\over \sqrt{s}}\right). 
\end{align*}
From this (b) and (c) follow easily. 
\end{proof} 
It is of methodological interest to have also a purely fluctuation-theoretic proof of Theorem 3.4., i.e., a proof which does not rely on the ``functional limit theorem'' approach used above. \black  The reviewer suggested the following alternative derivation of \ref{thmconv} 
based on Theorem 4.4.   Assume the conditions introduced in Section 3.  
\begin{thm} = {\bf Theorem 3.4}\black  
 $${\d^2 \over 2\s(\d)\black}Z_0^\d \lra A \ \mbox{ in distribution} \mbox{ as }\d\searrow 0. $$
\end{thm}
\begin{proof}
In principle, we follow the line of argument used for a similar proof in \cite{Shnwa}. Let $\eps>0$ and split the series
in the exponent of the right-hand side of \eqref{asc} into three parts:
$$\sum_{k=1}^\infty =\sum_{k=1}^{\eps/\d^2}\; + \; \sum_{\eps/\d^2}^{T/\d^2}\; + \;\sum_{T/\d^2}^\infty=\sum\nolimits_1 +\sum\nolimits_2 +\sum\nolimits_3\;.$$
Let $s>0$ and set $$z=e^{-s\d^2/2\s(\d)}.$$ 
We consider the different sums separately, starting with  $\sum\nolimits_1$:
$$\sum_{k=0}^{\eps/\d^2} (1-z^k) \frac{\PP(S^\d_k<0)}{k}\leq \frac{s\d^2}{2\s(\d)} \sum_{k=0}^{\eps/\d^2} \PP(S^\d_k<0)\leq \frac{s\eps}{2\s(\d)}.$$
Furthermore, $\PP(S_k^\d<0)= \PP(\sum_{j=1}^k X_{\d,j}<-k\d)\leq \s(\d)/(k\d^2)$ by Chebyshev's inequality.
 Therefore we  obtain for $\eps>\d^2$ 
$$\sum_{k\geq \eps/\d^2}\frac{(1-z^k)}{k}\PP(S^\d_k<0) \leq \frac{\s(\d)}{\d^2}\sum_{k\geq \eps/\d^2}\frac{1}{k^2} \leq\frac{\s(\d)}{\d^2} \int_{\eps/\d^2}^\infty \frac{1}{(x-1)^2}\,dx=\frac{\s(\d)}{\eps-\d^2}.$$
Since $\sigma(\d)\lra \s\in(0,\infty)$ as $\d\lra 0$, there is a $\d_0$ such that $2\d_0^2<\eps$ and $\s(\d)$ is bounded for $\d\leq \d_0$. 
Without loss of generality assume in the sequel $\d\leq\d_0$. 
Then $\sum\nolimits_3$ can be made arbitrarily small by a suitable choice of $T$, and $\sum\nolimits_2\leq 2C/\eps$ for a suitable constant $C$. 

For $\sum\nolimits_2$ we use the asymptotic normality of $\d S^\d_{t/\d^2}$ (which is implied by the Lindeberg condition, see  the beginning of section 3)\black:
$$\PP(\d S_k^d<0) \lra \PP(N(t,\s t)<0)=\Phi(-\sqrt{\frac{t}{\s}})\; \mbox{ as }\; \d \lra 0,\,k\d^2\lra t$$
(uniformly for $t\in [\eps,T]$), and by the dominated convergence we conclude that 
 $$\sum\nolimits_2 \lra \int_\eps^T \frac{1-e^{-t/2\s}}{t}\, \Phi(-\sqrt{\frac{t}{\s}})\,dt.$$
 Letting $\eps\lra 0, T\lra \infty$ we finally arrive at
 \begin{align} 
\label{domi} 
\EE e^{-s\frac{\d^2}{2\s}}\lra \exp\{-\int_0^\infty \frac{1-e^{-su}}{u}\,\Phi(-\sqrt{2u})\,du\}.
\end{align} 
Evaluating the integral finishes the proof. Avoiding the calculation, it suffices to notice that the right side of (\ref{domi})  
is independent of the underlying 
distribution of the random walk so that one can look at the example of the normal random walk computed above, which leads to the 
conclusion that the right side of (\ref{domi})is equal to $2/(1+\sqrt{1+s})$. 
\end{proof}  

The advantage of this proof is that it generalizes to the $\a$-stable case ($1<\a<2$) essentially unchanged - the main difficulties (the corresponding  estimates
for these cases) can be overcome using inequality $(6)$ in \cite{Shnwa}. 

 We close this section with a few remarks on the simple random walk \black
taking step +1 with probability $p>1/2$ and step -1 with  probability $q=1-p$. It is well-known that in this example 
$$r(z)={1 -\sqrt{1-4pqz^2} \over 2pz},$$ 
so that a quick calculation shows that 
$$ \EE z^{Z_0}={1-r(1) \over 1-r(z)} \black
= \dfrac{(p-q)(1+\sqrt{ 1-4pqz^2})}{p(1-2z^2 +\sqrt{ 1-4pqz^2})}$$  

and  
$2(p-\frac{1}{2})^2Z_0 \longrightarrow A$ in distribution as $p\searrow 1/2$.

\begin{rem}  
 Let $T_0(p)=\sup\{n\ge 0\,:\,S_n^{(0)}=0\}$ the time of the last return to the origin. 
In the symmetric case $p=1/2$  the walk is persistent and  $T_0(1/2)=\infty$ almost surely.  In the transient case $p>1/2$,  
 $T_0(p)$ has generating function 
$$h(z)={ p-q \over \sqrt{1-4pqz^2}}.$$
A short computation yields that ${1 \over 2}(p-q)^2T_0(p)$ converges in distribution as $p\searrow 1/2$, the limiting distribution 	
having Laplace transform $\dfrac{1}{\sqrt{ 1+s}}$, i.e., being the $\Gamma_{1,{1 \over 2}}$ distribution with density $\gamma_{1,{1 \over 2}}(t)$ as above.
\qed 
\end{rem}
  \begin{rem}     
Let $N_0(p)$ denote the number of zeros of the random walk. Then
$$\PP(N_0(p)=r,\,T_0(p)=2n)={r \over n-r}{2n- r \choose n} 2\,(pq)^n$$
and $(\delta N_0(p),{1 \over 2}\delta^2 T_0(p))$ converges weakly to the distribution with density
$$f(y,t)=1_{(0,\infty)}(y)\,1_{(0,\infty)}(t)\,{y \over 2t}{1 \over \sqrt{2\pi t}}\, e^{-({y^2 / 4t})\,-t}\;\;.$$
In particular,  $\delta N_0(p)$ is asymptotically $\exp(1)$. 
 For the symmetric random walk let $N_{0, 2n}$ denote the number of zeros up to time $2n$. A classical theorem of Chung-Hunt \cite{ChH} states that 
$\sqrt{{2 / n}} N_{0, 2n} $  is asymptotically distributed as $
|N(0,1)|$. All these results show that deviations from the symmetric random walk become clearly visible after $n\approx \delta^{-2}$ steps. 
 While characteristics like the positive sojourn time and the last exit time from zero are in both cases of approximately the same size their distributions differ. 
For the last exit time from zero a precise description is given in Theorem 2.3.  \black. 
\qed 
\end{rem}  

Apparently the distribution of $A$ occurs naturally as a limit of occupation times for random walks with drift.
It is well-known (see e.g. Section XIV.3 in \cite{Fel2}) that the deeper reason for the frequent occurrence of the (generalized) arcsine distributions 
lies in their intimate connection to distribution functions with regularly varying tails. The same explanation applies here.
 In the case of drift zero the distribution functions of the ladder epochs are attracted to the standard positive stable distribution
of index 1/2 and the positive (negative) sojourn times are asymptotically arcsine-distributed. In the cases with small drift (and finite variance) the ladder epochs are
attracted to an associated distribution of this stable distribution, and therefore the positive (negative) sojourn times have asymptotically the distribution of $A$. 

{\bf Acknowledgement}. We would like to thank the referee for valuable remarks and in particular for suggesting the alternative proof
of Theorem \ref{thmconv} given in Section 4.
\black

\end{document}